\documentclass{amsart}   	
\usepackage{amssymb,amsthm}
\usepackage{amsthm}
\usepackage{yhmath}
\usepackage{graphicx}
\usepackage{enumitem}
\usepackage{tikz}
\usepackage{tikz-cd}
\usetikzlibrary{calc}
\usetikzlibrary{decorations.pathmorphing}
\usepackage{hyperref}
\usepackage{adjustbox}
\usepackage{mathtools}
\usepackage[all]{xy}
\DeclarePairedDelimiter\ceil{\lceil}{\rceil}

\tikzset{curve/.style={settings={#1},to path={(\tikztostart)
    .. controls ($(\tikztostart)!\pv{pos}!(\tikztotarget)!\pv{height}!270:(\tikztotarget)$)
    and ($(\tikztostart)!1-\pv{pos}!(\tikztotarget)!\pv{height}!270:(\tikztotarget)$)
    .. (\tikztotarget)\tikztonodes}},
    settings/.code={\tikzset{quiver/.cd,#1}
        \def\pv##1{\pgfkeysvalueof{/tikz/quiver/##1}}},
    quiver/.cd,pos/.initial=0.35,height/.initial=0}

\tikzset{tail reversed/.code={\pgfsetarrowsstart{tikzcd to}}}
\tikzset{2tail/.code={\pgfsetarrowsstart{Implies[reversed]}}}
\tikzset{2tail reversed/.code={\pgfsetarrowsstart{Implies}}}
\tikzset{no body/.style={/tikz/dash pattern=on 0 off 1mm}}

\graphicspath{ {./images/} }

\newtheorem{theorem}{Theorem}[section]
\newtheorem{proposition}[theorem]{Proposition}
\newtheorem{lemma}[theorem]{Lemma}
\newtheorem{corollary}[theorem]{Corollary}

\theoremstyle{definition}

\newcommand{\qedno}{\null\nobreak\hfill\ensuremath{\square}} 
\newcommand{\uxa}{\ensuremath{(\underline{X},\underline{A})}}

\newcommand{\caa}{\ensuremath{(\underline{CA},\underline{A})}}
\newcommand{\cxx}{\ensuremath{(\underline{CX},\underline{X})}}

\newcommand{\zk}{\ensuremath{\mathcal{Z}_{K}}}

\usepackage[utf8]{inputenc}

\title{Loop space decompositions of moment-angle complexes associated to two dimensional simplicial complexes}
\author{Lewis Stanton}

\begin{document}

\begin{abstract}
We show that the loop space of a moment-angle complex associated to a $2$-dimensional simplicial complex decomposes as a finite type product of spheres, loops on spheres, and certain indecomposable spaces which appear in the loop space decomposition of Moore spaces. We also give conditions on certain subcomplexes under which, localised away from sufficiently many primes, the loop space of a moment-angle complex decomposes as a finite type product of spheres and loops on spheres.
\end{abstract}

\maketitle

\section{Introduction}

Polyhedral products are a natural subspace of the Cartesian product which are indexed by the face poset of a simplicial complex. They have generated much interest due to their far reaching applications across mathematics (see \cite{BBC}). Let $K$ be a simplicial complex on the vertex set $[m]$, and for $1 \leq i \leq m$, let $(X_i,A_i)$ be a pair of pointed $CW$-complexes, where $A_i$ is a pointed $CW$-subcomplex of $X_i$. The \textit{polyhedral product associated to} $K$ is \[\uxa^K = \bigcup\limits_{\sigma \in K} \left(\prod\limits_{i=1}^m Y^\sigma_i\right),\] where $Y^\sigma_i = X_i$ if $i \in \sigma$, and $Y^\sigma_i = A_i$ if $i \notin \sigma$. An important special case, which appears in toric topology, is when $(X_i,A_i) = (D^2,S^1)$ for all $i$. These polyhedral products are called \textit{moment-angle complexes}, and are denoted $\mathcal{Z}_K$. 

One particular problem associated to moment-angle complexes is understanding their loop spaces. In the case that $K$ is a flag complex, the loop homology of moment-angle complexes models commutator subalgebras of algebraic analogues of right angled Coxeter groups \cite{GPTW}. More generally, for any simplicial complex $K$, the loop space of the corresponding moment-angle complex is related to a certain diagonal subspace arrangement \cite{D}. When $K$ is flag, most homotopical and homological information about $\Omega \zk$ is known. In particular, a coarse description of $\Omega \zk$ was given in \cite{S}, which was upgraded to an explicit decomposition in \cite{V}. This allowed for a complete description of $H_*(\Omega \zk;R)$ as an algebra, where $R$ is any commutative ring with unit \cite{V}. Another interesting case is when $K$ is a $1$-dimensional simplicial complex (a graph). In this case, it was shown in \cite{S} that $\Omega \zk$ decomposes as a finite type product of spheres and loops on spheres. In particular, this implies that the homology of $\Omega \zk$ is torsion free. In this paper, we study the case of a $2$-dimensional simplicial complex, and give a coarse description of $\Omega \zk$ in this case. The homology of $\Omega \zk$ in this case can contain torsion, and this will require the introduction of certain indecomposable torsion spaces which have been considered in \cite{CMN1,CMN2,C}. We also give a coarse description of $\Omega \zk$ as a product of spheres and loops on spheres after localising away from a finite set of primes, under conditions on the rational cohomology of certain full subcomplexes of $K$. 

It is useful to identify two families of $H$-spaces. For a collection of topological spaces $\mathcal{X}$, let $\prod \mathcal{X}$ be the collection of spaces homotopy equivalent to a finite type product of spaces in $\mathcal{X}$. Let $\mathcal{P} := \{S^1,S^3,S^7, \Omega S^n \: | \: n \geq 2, n \notin\{2,4,8\}\}$. In \cite{S}, it was shown that $\prod \mathcal{P}$ is closed under retracts, and this was the key ingredient in proving coarse descriptions of $\Omega \zk$, in the case that $K$ is the $k$-skeleton of a flag complex. In this paper, we extend this to include torsion spaces. Denote the mod $p^r$ Moore space by $P^n(p^r)$, which is the mapping cone of the degree $p^r$ map on $S^{n-1}$. By \cite[Theorem 1.1]{HW}, for a finite type $H$-space $X$ localised at a prime $p$, there is a unique decomposition of $X$, up to homotopy, as a finite type product of indecomposable spaces. Let $\mathcal{T}$ be the collection of indecomposable spaces which appear in the decomposition of the loop space of a wedge of Moore spaces of the form $\bigvee_{i=1}^m P^{n_i}(p_i^{r_i})$, where $m \geq 2$, $n_i \geq 3$, $p_i$ is a prime and $r_i \geq 1$. Through an adaptation of the argument in \cite{S}, we will show in Section ~\ref{sec:closureofPandT} that $\prod (\mathcal{P} \cup \mathcal{T})$ is also closed under retracts. This is the key technical result which is required to prove the main result of this paper. \begin{theorem}[Theorem ~\ref{thm:2dimloop} in the text]\label{thm:2dimintro}
 Let $K$ be a $2$-dimensional simplicial complex. Then $\Omega \mathcal{Z}_{K} \in \prod (\mathcal{P} \cup \mathcal{T})$.
\end{theorem} 

In Section ~\ref{sec:closureofWcupM}, a collection of co-$H$ spaces, $\bigvee (\mathcal{W} \cup \mathcal{M})$, related to $\prod (\mathcal{P} \cup \mathcal{T})$ will be defined, and it will be shown that this collection is also closed under retracts. An important ingredient of the proof of Theorem ~\ref{thm:2dimintro} is a generalisation of \cite[Theorem 1.1]{S}. For a simplicial complex $K$, let $C_K$ be the collection of full subcomplexes of $K$ whose $1$-skeleton has no missing edges. The generalisation states $\Omega \zk$ being in $\prod (\mathcal{P} \cup \mathcal{T})$ depends only on $\Omega \mathcal{Z}_{K_I}$ being in $\prod(\mathcal{P} \cup \mathcal{T})$ for each full subcomplex $K_I \in C_K$ (see Theorem ~\ref{thm:inPcupTiffullsubcompleteskel}). This result can also be used to prove a localised result, which gives conditions on the rational cohomology of each $K_I \in C_K$ in $K$, for which after localising away from a finite set of primes (controlled by these full subcomplexes), $\Omega \zk \in \prod\mathcal{P}$. A simplicial complex $K$ is called $k$-neighbourly if any $I \subseteq [m]$ with $|I| \leq k+1$ spans a simplex. 

\begin{theorem}[Theorem ~\ref{thm:localisedprodofspheres} in the text]
\label{thm:introlocalisedprodofspheres}
    Let $K$ be a simplicial complex such that all cup products and higher Massey products in $H^*(|K_I|;\mathbb{Q})$ are trivial, for all $K_I \in C_K$. For $K_I \in C_K$, suppose that $K_I$ is $k_I$-neighbourly, and let $a = \max_{K_I \in C_K}\{|I|+dim(K_I)-2k_I\}$. Localise away from primes $p \leq \frac{1}{2}a$ and primes $p$ appearing as $p$-torsion in $H_*(|K_I|;\mathbb{Z})$ for any $K_I \in C_K$. Then $\Omega \mathcal{Z}_K \in \prod\mathcal{P}$.
\end{theorem}

Theorem ~\ref{thm:introlocalisedprodofspheres} has consequences for a question posed by McGibbon and Wilkerson. A space $X$ is called \textit{rationally elliptic} if it has finitely many rational homotopy groups. Otherwise, it is called \textit{rationally hyperbolic}. It was shown by McGibbon and Wilkerson \cite{MW} that if $X$ is rationally elliptic then at almost all primes $p$, the Steenrod algebra acts trivially on $H^*(\Omega X;\mathbb{Z}/p\mathbb{Z})$. They asked to what extent this holds for rationally hyperbolic spaces. We will show in Section ~\ref{subsec:localisedresult} that Theorem ~\ref{thm:introlocalisedprodofspheres} gives infinitely many examples for which this question has an affirmative answer. 

Theorem ~\ref{thm:2dimintro} and Theorem ~\ref{thm:introlocalisedprodofspheres} also verifies a conjecture of Anick \cite{A}. Anick conjectured that if $X$ is a finite, simply connected $CW$-complex, then at all but finitely many primes, $\Omega X \in \prod(\mathcal{P} \cup \mathcal{T})$. Theorem ~\ref{thm:2dimintro} and Theorem ~\ref{thm:introlocalisedprodofspheres} shows that such a decomposition holds for a family of moment-angle complexes.

We remark that the proofs in this paper hold more generally for polyhedral products of the form $\cxx^K$, where each $\Sigma X_i$ is homotopy equivalent to a finite type wedge of spheres and Moore spaces, however, we work only with moment-angle complexes to ease notation.

The author would like to thank Stephen Theriault for reading a draft of this work and providing many useful comments which helped to improve the paper. The author would also like to thank the anonymous referee for numerous helpful comments which improved the exposition of the paper.

\section{Preliminary results}
\subsection{Unique decomposition of \texorpdfstring{$H$}{H}-spaces and co-\texorpdfstring{$H$}{H} spaces}

In this subsection, we show a cancellation result after localisation that will be required to show that $\prod(\mathcal{P} \cup \mathcal{T})$ is closed under retracts. We will also require analogous results for wedges of spaces. We first state a result of \cite[Theorem 1.1]{HW} showing that after localising at a prime $p$, there is a unique decomposition of $H$-spaces and co-$H$ spaces into indecomposable spaces.

\begin{proposition}
\label{prop:uniquedecomp} The following hold:
\begin{enumerate}
    \item Let $X$ be a connected, finite type, $p$-local $H$-space. Then $X$ can be uniquely decomposed into a weak product of indecomposable factors up to order and homotopy equivalence.
    \item Let $X$ be a connected, finite type, $p$-local co-$H$-space. Then $X$ can be uniquely decomposed into a finite type wedge of indecomposable factors up to order and homotopy equivalence.\qed
\end{enumerate} 
\end{proposition}

We now apply Proposition ~\ref{prop:uniquedecomp} to show a cancellation result after localising at a prime $p$. Let $X$ and $Y$ be spaces. We say that $X$ \textit{retracts off} $Y$ if there exist maps $f:X \rightarrow Y$ and $g:Y \rightarrow X$ such that $g \circ f$ is homotopic to the identity on $X$.

\begin{proposition}
    \label{prop:cancellation}
    The following hold:
    \begin{enumerate}
        \item Let $X$ be a connected, finite type $H$-space, which is $p$-locally homotopy equivalent to a product \[X \simeq \prod\limits_{i \in \mathcal{I}} X_i,\] where each $X_i$ is indecomposable. If $A$ is a space which retracts off $X$, then there is a $p$-local homotopy equivalence \[A \simeq \prod\limits_{j \in \mathcal{J}} X_{j}\] where $\mathcal{J} \subseteq \mathcal{I}$.
        \item Let $X$ be a connected, finite type co-$H$-space, which is $p$-locally homotopy equivalent to a wedge \[X \simeq \bigvee\limits_{i \in \mathcal{I}} X_i,\] where each $X_i$ is indecomposable. If $A$ is a space which retracts off $X$, then there is a $p$-local homotopy equivalence \[A \simeq \bigvee\limits_{j \in \mathcal{J}} X_{j}\] where $\mathcal{J} \subseteq \mathcal{I}$
    \end{enumerate}
\end{proposition}
\begin{proof}
    We prove part $(1)$, and part $(2)$ follows by arguing dually. Localise at a prime $p$. Since $A$ retracts off $X$, there exists a map $g:X \rightarrow A$ which has a right homotopy inverse.  Proposition ~\ref{prop:uniquedecomp} implies that there is a unique $p$-local decomposition $A_i \simeq \prod_{k \in \mathcal{K}} A_k$, where each $A_k$ is indecomposable. Since $g$ has a right homotopy inverse and $X$ is an $H$-space, there is a $p$-local homotopy equivalence $X \simeq A \times F$, where $F$ is the homotopy fibre of $g$. The space $F$ retracts off $X$, and so $F$ is an $H$-space, implying by Proposition ~\ref{prop:uniquedecomp} that there is a $p$-local homotopy equivalence $F \simeq \prod_{k' \in \mathcal{K}'} F_{k'}$, where each $F_{k'}$ is indecomposable. Hence, \[\prod\limits_{i \in \mathcal{I}} X_i \simeq X \simeq A \times F \simeq  \prod_{k \in \mathcal{K}} A_k \times \prod_{k' \in \mathcal{K}'} F_{k'}.\] Since the product decomposition of $X$ is unique, there exists an indexing set $\mathcal{J} \subseteq \mathcal{I}$ such that $A \simeq \prod\limits_{j \in \mathcal{J}} X_{j}$.
\end{proof}

\subsection{Rational and \texorpdfstring{$p$}{p}-local decompositions of moment-angle complexes}

In this subsection, we state some preliminary localised decompositions of spaces. The first states conditions under which a finite $CW$-complex decomposes as a wedge of spheres after localising away from sufficiently many primes. This result is a mild generalisation of a result proved in \cite[Lemma 5.1]{HT}, however, the proof goes through unchanged.

\begin{lemma}
\label{lem:plocallywedge}
    Let $X$ be a simply-connected, finite $CW$-complex of dimension $d$ and connectivity $s$. Suppose that $X$ is rationally homotopy equivalent to a wedge of spheres. Let $p$ be a prime such that $p > \frac{1}{2}(d-s+1)$, and $H_*(X;\mathbb{Z})$ is $p$-torsion free. Then $X$ is $p$-locally homotopy equivalent to a wedge of spheres. \qedno
\end{lemma}

The next result relates to a rational decomposition for certain moment-angle complexes. Let $K$ be a simplicial complex. The simplicial complex $K$ is said to be \textit{Golod over a ring} $R$ if all products and higher Massey products in $H^*(\zk;R)$ are trivial. In this case, there is a relation between the rational Golodness of $K$ and $\zk$ being a suspension. The following result is attributed to Berglund, however, the reference now appears to be unavaliable. Therefore, we provide an alternative proof.

\begin{proposition}
    \label{prop:rationalgolodmeanssuspension}
    Let $K$ be a simplicial complex. Then $K$ is rationally Golod if and only if $\zk$ is rationally a co-$H$ space.
\end{proposition}
\begin{proof}
    Rationally, any co-$H$ space is homotopy equivalent to a wedge of spheres. Therefore, if $\zk$ is rationally a co-$H$ space then $K$ is rationally Golod. Suppose $K$ is rationally Golod. Then by \cite[Proposition 3.6]{K}, $K$ is Golod over $\mathbb{Z}/p\mathbb{Z}$, when $p$ is a sufficiently large prime. Moreover, by \cite[Theorem 3.1]{BG}, localised at a sufficiently large prime $p$, $\zk$ is a co-$H$ space if and only if $\zk$ is Golod over $\mathbb{Z}/p\mathbb{Z}$. Therefore, $\zk$ is rationally a co-$H$ space.
\end{proof}

A counterexample to the integral analogue of Proposition ~\ref{prop:rationalgolodmeanssuspension} was constructed in \cite{IY}. Since any co-$H$ space is rationally homotopy equivalent to a wedge of spheres, a rational decomposition of $\zk$ in this case can be recovered from its homology. For any moment-angle complex, a suspension splitting was proved in \cite[Theorem 2.21]{BBCG1}.

\begin{proposition}
    \label{prop:suspensionsplitting}
    Let $K$ be a simplicial complex. There is a homotopy equivalence \[\Sigma \zk \simeq \bigvee\limits_{I \notin K} \Sigma^{2+|I|} |K_I|.\eqno\qed\]
\end{proposition}

Finally, we require a result which relates the rational homotopy type and the $p$-local homotopy type of an $H$-space. This result is known as the Sullivan arithmetic square (see \cite[Theorem 8.1.3]{MP} for a modern presentation). For a prime $p$, denote by $X_{(p)}$ the localisation of $X$ at $p$, and let $X_{\mathbb{Q}}$ denote the rationalisation of $X$.

\begin{theorem}
    \label{thm:Sullivansquare}
    Let $X$ be an $H$-space. Then there is a homotopy pullback \[\begin{tikzcd}
	X & {\prod\limits_{p}X_{(p)}} \\
	{X_{\mathbb{Q}}} & {\prod\limits_{p} X_{\mathbb{Q}}.}
	\arrow[from=1-1, to=1-2]
	\arrow[from=1-1, to=2-1]
	\arrow[from=1-2, to=2-2]
	\arrow[from=2-1, to=2-2]
\end{tikzcd}\] In particular, if $X$ is rationally trivial, there is a homotopy equivalence \[X \simeq \prod\limits_{p} X_{(p)}.\eqno\qed\]
\end{theorem}

\section{Closure of \texorpdfstring{$\bigvee(\mathcal{W} \cup \mathcal{M})$}{W cup M} under retracts}
\label{sec:closureofWcupM}

To prove Theorem ~\ref{thm:2dimintro}, we will need to consider wedge decompositions of certain spaces. For a collection of topological spaces $\mathcal{X}$, let $\bigvee \mathcal{X}$ denote the collection of spaces homotopy equivalent to a finite type wedge of spaces in $\mathcal{X}$. Let $\mathcal{W}$ be the collection of simply connected spheres. By Proposition ~\ref{prop:uniquedecomp}, when localised at a prime, there is a unique decomposition of any co-$H$ space, up to homotopy, as a finite type wedge of indecomposable spaces. Let $\mathcal{M}$ be the collection of Moore spaces of the form $P^{n}(p^r)$, where $n \geq 3$, $p$ is a prime, and $r \geq 1$, and the indecomposable factors which appear as wedge summands in the unique $2$-local wedge decomposition of spaces of the form $\Sigma ((P^{n_1}(2) \wedge \cdots \wedge (P^{n_l}(2))$, where $l \geq 2$, and each $n_i \geq 3$. Note that we do not require smash products of Moore spaces of the form $P^n(p^r)$ when $p^r \neq 2$, since in this case by \cite[Corollary 6.6]{N3}, there is a homotopy equivalence \[P^n(p^{r_1}) \wedge P^m(p^{r_2}) \simeq P^{n+m}(p^{\min\{r_1,r_2\}}) \vee P^{n+m-1}(p^{\min\{r_1,r_2\}})\] when $p^{r_1},p^{r_2} \neq 2$. In general, the indecomposable wedge summands that appear in the decomposition of spaces of the form  $\Sigma ((P^{n_1}(2) \wedge \cdots \wedge (P^{n_l}(2)))$ are unknown, but some progress has been made in \cite{W}. In this section, we will show that $\bigvee(\mathcal{W} \cup \mathcal{M})$ is closed under retracts. This is well known for spaces in $\bigvee\mathcal{W}$, and it was shown for a wedge of Moore spaces of a fixed odd prime power in \cite[Lemma 4.2]{N2}. The introduction of the $2$-torsion spaces necessitates a more technical argument which we complete here.

Let $X \in \bigvee(\mathcal{W} \cup \mathcal{M})$, and let $A$ be a space which retracts off $X$. The strategy to show that $A \in \bigvee(\mathcal{W} \cup \mathcal{M})$ is to retract a sphere off $A$ for every $\mathbb{Z}$ summand which appears in the homology of $A$. This will give us a homotopy equivalence $A \simeq W \vee A'$, where $W \in \bigvee\mathcal{W}$, and the homology of $A'$ is torsion. We can then use Theorem ~\ref{thm:Sullivansquare} and Proposition ~\ref{prop:cancellation} to show that $A' \in \bigvee\mathcal{M}$. 

To retract spheres off $A$, we argue similarly to \cite[Section 3]{S}. Since $A$ retracts off $X$, there exists maps $f:A \rightarrow X$ and $g:X \rightarrow A$ such that $g \circ f \simeq id_A$. Define $\phi:X \rightarrow X$ as the composite \[\phi: X \xrightarrow{g} A \xrightarrow{f} X.\] Note that $\phi$ is an idempotent. Let $W$ be the wedge of spheres that appear in the wedge decomposition of $X$. Define $\phi'$ to be the composite \[\phi': W \hookrightarrow X \xrightarrow{\phi} X \xrightarrow{p} W,\] where the lefthand map is the inclusion, and $p$ is the pinch map. While $\phi'$ may not be an idempotent, the following shows that the induced map $(\phi')_*$ is an idempotent on homology, which suffices for our purposes. This follows from the following technical lemma.

\begin{lemma}
\label{lem:idempotentmapfreepart}
Let $G$ be a finitely generated abelian group, $G_{free}$ be the free part of $G$, and $G_{tor}$ be the torsion part of $G$. Let $\phi:G \rightarrow G$ be an idempotent. Then the composite \[\phi':G_{free} \xrightarrow{i} G \xrightarrow{\phi} G \xrightarrow{\pi} G_{free},\] where $i$ is the inclusion and $\pi$ is the projection, is an idempotent. Moreover, if $\phi(g,t) = (g',t')$ where $g,g' \in G_{free}$ and $t,t' \in G_{tor}$, then $g' \in \text{Im}(\phi')$.
\end{lemma}
\begin{proof}
 Consider $\phi' \otimes \mathbb{Q}$. Since $G_{free}$ is free, the maps $i \otimes \mathbb{Q}$ and $\pi \otimes \mathbb{Q}$ are both the identity map. By assumption, $\phi$ is idempotent implying that $\phi \otimes \mathbb{Q}$ is idempotent and so $\phi'\otimes \mathbb{Q}$ is idempotent. Hence, $\phi'$ is idempotent. A similar argument shows that the second part is true.
\end{proof}

The inclusion $W \rightarrow X$ and the pinch map $X \xrightarrow{p} W$ induce the inclusion and projection respectively of the free part of $H_*(X)$, and so by Lemma ~\ref{lem:idempotentmapfreepart}, $(\phi')_*$ is an idempotent on homology. Let $a \in H_n(A)$ be a generator of a $\mathbb{Z}$ summand. We aim to show that $S^n$ retracts off $A$. First, we show that $(\phi'_*)_n:H_n(W) \rightarrow H_n(W)$ is non-zero.

\begin{lemma}
    \label{lem:wedgenonzeromap}
    Let $a \in H_n(A)$ be a generator of a $\mathbb{Z}$ summand. Then the induced map $(\phi'_*)_n$ in homology is non-zero.
\end{lemma}
\begin{proof}
    Recall that $f:A \rightarrow X$ and $g:X \rightarrow A$ are maps such that $g \circ f \simeq id_A$, and $\phi$ is the composite \[\phi: X \xrightarrow{g} A \xrightarrow{f} X.\] Since $f_n$ is injective, $f_n(a)$ is non-zero and not torsion. Moreover, the composite $g_n \circ f_n$ is the identity, and so $g_n \circ f_n(a)$ is non-zero. Hence, $\phi_n(f_n(a)) = (f_n \circ g_n \circ f_n)(a) = f_n(a)$ is non-zero, and so Lemma ~\ref{lem:idempotentmapfreepart} implies that $(\phi'_*)_n$ is non-zero.
\end{proof}

Let $x$ be a generator of $\mathrm{Im}((\phi'_*)_n)$. Let $H_n(W) = \bigoplus_{i=1}^m \mathbb{Z}$, and write $x$ as $x=(x_1,\cdots,x_m)$. Since $(\phi'_*)_n$ is an idempotent of free abelian groups, $x$ must extend to a basis of $\mathbb{Z}^m$, and so the greatest common divisor of $x_1,\cdots,x_m$ is $1$ (see \cite[Lemma 2.3]{S} for example). B\'ezout's Lemma implies that there exists $y_1,\cdots,y_m \in \mathbb{Z}$ such that $\sum_{i=1}^m y_i x_i = 1$. Let $\rho:S^n \rightarrow W$ be the composite \[\rho:S^n \xrightarrow{\sigma} \bigvee\limits_{i=1}^m S^n \xrightarrow{\bigvee\limits_{i=1}^m p_{x_i}} \bigvee\limits_{i=1}^m S^n \hookrightarrow W,\] where $\sigma$ is a choice of $m$-fold suspension comultiplication, and $p_{x_i}$ is the $x_i^{th}$ degree map. Let $\gamma$ be a generator of $H_n(S^n)$. By definition of $\rho$ and the fact that $p_{x_i}$ induces multiplication by $x_i$ in homology, $\rho_n$ sends $\gamma$ to $x \in H_n(W)$. Now let $\rho'$ be the composite \[\rho':W \xrightarrow{p} \bigvee\limits_{i=1}^m S^n \xrightarrow{\bigvee\limits_{i=1}^m p_{y_i}} \bigvee\limits_{i=1}^m S^n \xrightarrow{\nabla} S^n ,\] where $p$ is the pinch map, and $\nabla$ is the fold map. By definition of $\rho'$, since the degree map $p_{y_i}$ induces multiplication by $y_i$ in homology, $\rho'_n$ sends $x$ to $\gamma$. Since $(\phi'_*)_n$ is idempotent, it fixes its image, and so the composite \[e:S^n \xrightarrow{\rho} W \xrightarrow{\phi'} W \xrightarrow{\rho'} S^n,\] is an isomorphism in homology. This implies that $e$ is a homotopy equivalence. Since $\phi'$ factors through $A$, $S^n$ retracts off $A$. Arguing dually to \cite[Proposition 3.3, Theorem 3.10]{S} for each generator of a $\mathbb{Z}$ summand in $H_*(A)$, we obtain the following.

\begin{proposition}
    \label{prop:retractingspheresoffA}
    Let $X \in \bigvee(\mathcal{W} \cup \mathcal{M})$, and let $A$ be a space which retracts off $X$. Then there is a homotopy equivalence \[A \simeq S \vee A',\] where $S \in \bigvee\mathcal{W}$, and the homology of $A'$ is torsion. \qedno
\end{proposition}

It now suffices to show that $A'$ in Proposition ~\ref{prop:retractingspheresoffA} is in $\mathcal{M}$. 

\begin{proposition}
    \label{prop:retractingMoorespacesoffA'}
    Let $X \in \bigvee(\mathcal{W} \cup \mathcal{M})$, and let $A'$ be a space which retracts off $X$, such that the homology of $A'$ is torsion. Then $A' \in \bigvee\mathcal{M}$.
\end{proposition}
\begin{proof}
    By assumption, $H_*(A')$ is torsion, and so $A'$ is rationally trivial. By Theorem ~\ref{thm:Sullivansquare}, there is a homotopy equivalence \[A' \simeq \prod\limits_{p} A'_{(p)}.\] For each prime $p$, $A'_{(p)}$ retracts off $X_{(p)}$. Proposition ~\ref{prop:cancellation}(2) therefore implies that each $A'_{(p)}$ is homotopy equivalent to a wedge of $p$-torsion spaces in $\mathcal{M}$.

    Let $i:\bigvee_{p} A'_{(p)}\rightarrow \prod_{p} A'_{(p)}$ be the inclusion. Localised at a prime $p$, the map $i$ is the identity map $A'_{(p)}\rightarrow A'_{(p)}$ and so $i$ is a homotopy equivalence localised at any prime $p$. Rationally each $A'_{(p)}$ is contractible, and so $i$ is a homotopy equivalence rationally. Hence, $i$ is a homotopy equivalence integrally. Putting this all together, we obtain a homotopy equivalence $A' \simeq \bigvee_{p} A'_{(p)}$, where each $A'_{(p)} \in \bigvee\mathcal{M}$.
\end{proof}

Combining Proposition ~\ref{prop:retractingspheresoffA} and Proposition ~\ref{prop:retractingMoorespacesoffA'}, we obtain the following. 

\begin{theorem}
    \label{thm:WcupMclosedunderretracts}
    Let $X \in \bigvee(\mathcal{W} \cup \mathcal{M})$, and let $A$ be a space which retracts off $X$. Then $A \in \bigvee(\mathcal{W} \cup \mathcal{M})$. \qedno
\end{theorem}

\section{Closure of \texorpdfstring{$\prod(\mathcal{P} \cup \mathcal{T})$}{P cup T} under retracts}
\label{sec:closureofPandT}
\subsection{Special cases}
Recall that for a collection of topological spaces $\mathcal{X}$, $\prod \mathcal{X}$ is the collection of spaces homotopy equivalent to a finite type product of spaces in $\mathcal{X}$. Moreover, recall the collections $\mathcal{P} := \{S^1,S^3,S^7, \Omega S^n \: | \: n \geq 2, n \notin\{2,4,8\}\}$, and $\mathcal{T}$, which is the collection of indecomposable spaces which appear in the decomposition of the loop space of a wedge of Moore spaces of the form $\bigvee_{i=1}^m P^{n_i}(p_i^{r_i})$, where $m \geq 2$, $n_i \geq 3$, $p_i$ is a prime and $r_i \geq 1$. In this section, we show that $\prod (\mathcal{P} \cup \mathcal{T})$ is closed under retracts. 

We start with some special cases. First, we have the following result from \cite[Theorem 3.10]{S}. 
\begin{theorem}
\label{thm:Pclosedunderretracts}
    Let $X \in \prod\mathcal{P}$, and $A$ be a space which retracts off $X$. Then $A \in \prod\mathcal{P}$. \qedno
\end{theorem}

We can also prove a similar result in the case of a space $A$ retracting off $X \in \prod(\mathcal{P} \cup \mathcal{T})$, where the homology of $A$ is torsion.

\begin{theorem}
\label{thm:Tclosedunderretracts}
    Let $X \in \prod(\mathcal{P} \cup \mathcal{T})$, and $A$ be a space which retracts off $X$, such that the homology of $A$ is torsion. Then $A \in \prod\mathcal{T}$.
\end{theorem}
\begin{proof}
    Since the homology of $A$ is torsion, $A$ is rationally trivial. Therefore by Theorem ~\ref{thm:Sullivansquare}, there is a homotopy equivalence \[A \simeq \prod\limits_{p} A_{(p)}.\] For each prime $p$, $A_{(p)}$ retracts off $X_{(p)}$. Proposition ~\ref{prop:cancellation} implies that each $A_{(p)} \in \prod\mathcal{T}$, and so $A \in \prod\mathcal{T}$.
\end{proof}

\subsection{Review of the proof of Theorem ~\ref{thm:Pclosedunderretracts}}
\label{subsec:setupgeneral}

First, we recall the strategy from \cite[Section 3]{S} which was used to prove Theorem ~\ref{thm:Pclosedunderretracts}. The strategy is similar to the one used in Section ~\ref{sec:closureofWcupM}. Let $X \in \prod\mathcal{P}$, and let $A$ be a space which retracts off $X$. This implies there are maps $f:A \rightarrow X$, and $g:X \rightarrow A$ such that $g \circ f$ is homotopic to the identity on $A$. The first ingredient of the proof is the idempotent $\phi: X \xrightarrow{g} A \xrightarrow{f} X$. The key property that is used here is that $\phi_*$ is an idempotent on homology, and so $\phi_*$ fixes its image. Let $n$ be an integer such that $H_n(A)$ contains a primitive generator. The proof is split into three cases, the first is where $n \in \{1,2,3,6,7,14,4m\:|\: m \geq 1\}$, the second is where $n = 4m+2$, $m \geq 2$, $m \neq 3$, and the third is where $n = 2m-1$, where $m \notin \{1,2,4\}$. 

Consider the first and third case (see \cite[Subsections 3.2 and 3.4]{S}), where \[n \in  \{1,2,3,6,7,14,4m, 2l-1\:|\: m,l \geq 1, l \notin \{1,2,4\}\}.\] Fix such an $n$ and write $X$ as \[X \simeq \prod\limits_{i=1}^{m_Y} Y_i \times \prod\limits_{\alpha' \in \mathcal{I}'} Z_{\alpha'},\] where each $Y_i$ is an instance of $S^{n}$ if $n \in \{1,3,7\}$, or $Y_i = \Omega S^{n+1}$ otherwise, and each $Z_{\alpha'}$ are the spheres and loops on spheres not equal to $Y_i$. Let $Y = S^n$ if $n \in \{1,3,7\}$, or $Y = \Omega S^{n+1}$ otherwise. In this case, the bottom non-vanishing degree of each $Y_i$ gives a basis of primitives $\{\gamma_1,\cdots,\gamma_{m_Y}\}$ of $H_n(X)$. It was shown that there exists a non-zero element $x= \sum_{i=1}^{m_Y} y_i \gamma_i \in \text{Im}(\phi_*)$ such that the greatest common divisor of $y_1,\cdots,y_{m_Y}$ is $1$. Let $\gamma$ be a generator of the lowest non-vanishing degree in the homology of $Y$. Two maps $\rho:Y \rightarrow X$, and $\rho':X \rightarrow Y$ were defined such that $\rho_*(\gamma) = x$, and $\rho'_*(x) = \gamma$. Since $\phi_*$ fixes its image, the composite \[e:Y \xrightarrow{\rho} X \xrightarrow{\phi} X \xrightarrow{\rho'} Y\] is an isomorphism on the lowest non-vanishing degree in homology. If $Y$ is a sphere, then $e$ is a homotopy equivalence by Whitehead's theorem. If $Y$ is the loops on a sphere, localisation and atomicity properties of the loops on spheres (with a slight adjustment to the maps $\rho$ and $\rho'$ in the case that $n = 2l-1$) are used to show that $e$ is a homotopy equivalence, implying that $Y$ retracts off $X$. 

Now consider the second case (see \cite[Subsection 3.3]{S}), where $H_{4n+2}(A)$ contains a primitive generator, $n \geq 2$, $n \neq 3$. In this case, write $X$ as \[X \simeq \prod\limits_{i=1}^{m_Y} \Omega S^{4n+3}_i \times \prod\limits_{j=1}^{m_{\overline{Y}}} \Omega S^{2n+2}_j \times \prod\limits_{\alpha' \in \mathcal{I}'} Z_{\alpha'}\] where each $Z_{\alpha'}$ are the spheres and loops on spheres that are not equal to $\Omega S^{4n+3}$ or $\Omega S^{2n+2}$. A basis of primitives $\{\gamma_1,\cdots,\gamma_{m_Y},\overline{\gamma}_{1},\cdots,\overline{\gamma}_{m_{\overline{Y}}}\}$ was obtained of $H_{4n+2}(X)$, where $\gamma_i$ is a generator of $H_{4n+2}(\Omega S^{4n+3}_i)$, and $\overline{\gamma}_i$ is a generator of $H_{4n+2}(\Omega S^{2n+2}_i)$. It was shown that there exists a non-zero element $\sum_{i=1}^{m_Y} y_i \gamma_i + \sum_{i=1}^{m_{\overline{Y}}} 2\overline{y}_i\overline{\gamma}_i \in \text{Im}(\phi_*)$ such that the greatest common divisor of $y_1,\cdots,y_{m_Y},2\overline{y}_1,\cdots,2\overline{y}_{m_{\overline{Y}}}$ is $1$, and as in the previous case, this element was used to define maps $\rho:Y \rightarrow X$, and $\rho':X \rightarrow Y$ such that the composite \[Y \xrightarrow{\rho} X \xrightarrow{\phi} X \xrightarrow{\rho'} Y\] is a homotopy equivalence, implying that $Y$ retracts off $X$. Therefore, for each primitive generator in $H_*(A)$, we can retract a sphere or the loops on a sphere off $A$. Iterating this, we obtain a product decomposition for $A$ as a product of spheres and loops on spheres (see \cite[Theorem 3.10]{S}).

To generalise to the case where $X \in \prod(\mathcal{P} \cup \mathcal{T})$, we first retract off all the spheres and loops on spheres that we expect to obtain in a decomposition for $A$, by analysing the coalgebra structure of $H_*(A;\mathbb{Q})$. This will give us a homotopy equivalence $A \simeq P \times A'$, where $P \in \prod\mathcal{P}$, and the homology of $A'$ is torsion. We can then use Theorem ~\ref{thm:Sullivansquare} to obtain that $A' \in \prod\mathcal{T}$. 

\subsection{Defining \texorpdfstring{$\phi'$}{phi'}}
\label{subsec:phidefinition}

From now on, homology will be assumed to be taken to be taken with integral coefficients unless otherwise stated. The coalgebra structure on the homology of a space is defined whenever the K\"unneth isomorphism holds. In this case however, the homology of $X$ and $A$ may contain torsion, and so we can not appeal to the coalgebra structure in order to repeat the argument for Theorem ~\ref{thm:Pclosedunderretracts}. However, we will adjust the map $\phi$ to obtain a self map $\phi'$ of the spheres and loops on spheres that appear in the product decomposition of $X$ which is idempotent in homology. This will allow us to find the required elements in the image of $\phi'$ in order to appeal to the argument in \cite{S}. Let $X \in \prod(\mathcal{P} \cup \mathcal{T})$, and let $A$ be a space which retracts off $X$, such that $H_*(A;\mathbb{Q})$ is non-trivial. In particular there exists maps $f:A \rightarrow X$ and $g:X \rightarrow A$ such that $g \circ f$ is homotopic to the identity on $A$. Observe that the map $\phi = f \circ g$ is an idempotent. Write $X$ as $X \simeq S \times M$, where $S \in \prod\mathcal{P}$ and $M \in \prod\mathcal{T}$. Define the map $\phi':S \rightarrow S$ as the composite \[\phi':S \hookrightarrow X \xrightarrow{\phi} X \xrightarrow{\pi} S,\] where the left map is the inclusion of $S$ into $X$, and $\pi$ is the projection. We would like $\phi'$ to also be an idempotent in order to emulate the map $\phi$ in the case where $X \in \prod\mathcal{P}$. This may not be true for the map itself, however, the inclusion $S \rightarrow X$ induces the inclusion of the free part of $H_*(X)$ and the projection $X \rightarrow S$ induces the projection onto the free part of $H_*(X)$. Therefore, Lemma ~\ref{lem:idempotentmapfreepart} implies that $(\phi')_*$ is an idempotent on homology.

To appeal to the argument in \cite[Section 3]{S}, we require a primitive element $x \in \text{Im}(\phi')$ with the properties as described in Subsection ~\ref{subsec:setupgeneral}. First, we need to show that in certain degrees, $(\phi')_*$ is non-zero when restricted to the submodule of primitives in $H_*(S)$. In the case that $X \in \prod\mathcal{P}$, this was done by showing that $\phi_n$ is non-zero when restricted to the submodule of primitives whenever there is a primitive generator in $H_n(A)$. However, in this case, we can not appeal to the coalgebra structure in integral homology as $H_*(A)$ may contain torsion. However, rational homology $H_*(A;\mathbb{Q})$ does have a coalgebra structure. We can use this to show that if $a \in H_n(A)$ is an element which reduces to a primitive generator in rational homology, then $(\phi'_*)_n$ is non-zero when restricted to the submodule of primitives in $H_n(S)$.

\begin{lemma}
    \label{lem:phinonzero}
Suppose that $a \in H_*(A)$ is a generator of a $\mathbb{Z}$ summand in degree $n$, which reduces to a primitive generator in rational homology. Then $f_*(a)$ maps to an element in $H_n(X)$ whose free part reduces to an element in the submodule of primitives in rational homology, and $(\phi'_*)_n$ is non-zero when restricted to the submodule of primitives in $H_*(S)$.
\end{lemma}
\begin{proof}
    Let $\overline{a} \in H_*(A;\mathbb{Q})$ be the primitive generator which is the reduction of $a$. Recall that $\phi'$ is the composite \[\phi':S \hookrightarrow X \xrightarrow{\phi} X \xrightarrow{\pi} S,\] where $S$ is the product of spheres and loops on spheres that appear in $X$, $\phi$ is the composite $f \circ g$ which is an idempotent, the left map is the inclusion of $S$ into $X$, and $\pi$ is the projection. By the naturality of the universal coefficient theorem with respect to coefficients, there is a commutative diagram \[\begin{tikzcd}
	{H_n(A)} & {H_n(X)} \\
	{H_n(A)\otimes \mathbb{Q}} & {H_n(X) \otimes \mathbb{Q}.}
	\arrow["{f_n}", from=1-1, to=1-2]
	\arrow[from=1-1, to=2-1]
	\arrow[from=1-2, to=2-2]
	\arrow["{f_n \otimes \mathbb{Q}}", from=2-1, to=2-2]
\end{tikzcd}\] In particular, since $H_n(A) \otimes \mathbb{Q}$ is a coalgebra, and $\overline{a}$ is a primitive generator, $f_n(a)$ must map to an element $x \in H_n(X)$ such that the free part of $x$, which we will denote by $x'$, reduces to a primitive element in $H_n(X) \otimes \mathbb{Q}$. This proves the first part of the lemma. Since $f_n(a)$ is injective, $x$ is non-zero and has infinite order, and so $x'$ is also non-zero. By definition of $\phi$, the image of $\phi$ is equal to the image of $f$, and so $x \in \text{Im}(\phi)$. Lemma ~\ref{lem:idempotentmapfreepart} implies that $x'$ is in the image of $(\phi'_*)_n$. Since $S \in \prod\mathcal{P}$ and $x'$ reduces to a primitive element in rational homology, $(\phi'_*)_n(x')$ is contained in the submodule of primitives of $H_n(S)$ integrally. Hence $(\phi'_*)_n$ is non-zero when restricted to the submodule of primitives in $H_*(S)$.
\end{proof}

\subsection{Case 1}

In this subsection, fix $n \in \{1,2,3,6,7,14,4m,2l-1\:|\: m,l \geq 1, l \notin \{1,2,4\}\}$ such that $A$ contains a generator $a \in H_n(A)$ which reduces to a primitive generator in rational homology. Write $S$ as \[S \simeq \prod\limits_{i=1}^{m_Y} Y_i \times \prod\limits_{\alpha' \in \mathcal{I}'} Z_{\alpha'},\] where each $Y_i$ is an instance of $S^{n}$ if $n \in \{1,3,7\}$, or $Y_i = \Omega S^{n+1}$ otherwise, and the factors $Z_{\alpha'}$ are the spheres, and loops on spheres that are not equal to $Y_i$. In this case, the bottom non-vanishing degree of each $Y_i$ gives a basis of primitives $\{\gamma_1,\cdots,\gamma_{m_Y}\}$ of $H_n(S)$. As in Subsection ~\ref{subsec:setupgeneral}, we require an element $\sum_{i=1}^{m_Y} y_i \gamma_i \in \text{Im}(\phi')$ such that the greatest common divisor of $y_1,\cdots,y_{m_Y}$ is $1$. By Lemma ~\ref{lem:phinonzero}, $(\phi'_*)_n$ is a non-zero idempotent map in this degree. Let $\lambda = \sum_{i=1}^{m_Y} z_i \gamma_i$ be a generator of $\text{Im}(\phi')$. Since $(\phi'_*)_n$ is an idempotent, the greatest common divisor of $z_1,\cdots,z_{m_Y}$ is $1$ (see for example \cite[Section 2]{S}). Hence $\lambda$ has the desired properties. Therefore, we can argue as in \cite[Subsection 3.2]{S} in the case where $n \in \{1,2,3,6,7,14,4m\:|\: m \geq 1\}$, and \cite[Subsection 3.4]{S} in the case where $n = 2l-1$, $l \notin \{1,2,4\}$ to retract the following spheres and loops on spheres off $A$: for each generator of a $\mathbb{Z}$ summand in $H_n(A)$ which reduces to a primitive generator in rational homology, an $S^n$ or $\Omega S^{n+1}$ retracts off $A$. We obtain the following.

\begin{lemma}
    \label{lem:case1restractspheres}
    Let $X \in \prod(\mathcal{P} \cup \mathcal{T})$, and $A$ be a space which retracts off $X$. There is a homotopy equivalence \[A \simeq P \times A',\] where $P$ is a product of spheres of the form $S^{n}$, where $n \in \{1,3,7\}$, and loops on spheres of the form $\Omega S^{m+1}$, where $m \in \{2,6,14,4m,2l-1\:|\: m \geq 1,l \notin \{1,2,4\}\}$. Moreover, the only generators in $H_*(A')$ which reduce to primitive generators in rational homology are in degrees of the form $4k+2$, where $k \geq 2$, $k \neq 3$. \qedno
\end{lemma}

\subsection{Case 2}

By Lemma ~\ref{lem:case1restractspheres}, it suffices to consider a space $A$ which retracts off $X \in \prod(\mathcal{P} \cup \mathcal{T})$, such that the only generators of $H_*(A)$ which reduce to primitive generators in rational homology are in degrees of the form $4n+2$, where $n \geq 2$, and $n \neq 3$. Fixing $n$, write $X$ as \[X \simeq \prod\limits_{i=1}^{m_Y} \Omega S^{4n+3}_i \times \prod\limits_{j=1}^{m_{\overline{Y}}} \Omega S^{2n+2}_j \times \prod\limits_{\alpha' \in \mathcal{I}'} Z_{\alpha'}\] where the factors $Z_{\alpha'}$ are the spheres, loops on spheres and indecomposable torsion spaces that are not equal to $\Omega S^{4n+3}$ or $\Omega S^{2n+2}$. For $1 \leq i \leq m_{Y}$, let $\gamma_i$ be the generator of $H_{4n+2}(X)$ corresponding to a generator of $H_{4n+2}(\Omega S^{4n+3}_i)$, and for $1 \leq j \leq m_{\overline{Y}}$, let $\overline{\gamma}_j$ be the generator of $H_{4n+2}(X)$ corresponding to a generator of $H_{4n+2}(\Omega S_j^{2n+2})$. Observe that by definition of $X$, the generators $\gamma_i$ and $\overline{\gamma}_j$ form a basis for the submodule of primitives of rational homology in degree $4n+2$. Let $a \in H_{4n+2}(A)$ be a generator of a $\mathbb{Z}$ summand which reduces to a primitive generator in rational homology. By Lemma ~\ref{lem:idempotentmapfreepart}, $f_*(a) = \sum_{i=1}^{m_Y} y_i\gamma_i + \sum_{j=1}^{m_{\overline{Y}}} \overline{y}_j \overline{\gamma}_j + t$, where $t$ has finite order. Recall $\phi'$ is the composite \[\phi':S \hookrightarrow X \xrightarrow{\phi} X \xrightarrow{\pi} S,\] where the left map is the inclusion of $S$ into $X$, $\phi = f \circ g$ and $\pi$ is the projection. Since $f_*$ is injective, the image of $\phi_*$ is equal to the image of $f_*$. Therefore, by Lemma ~\ref{lem:phinonzero}, $\sum_{i=1}^{m_Y} y_i\gamma_i + \sum_{j=1}^{m_{\overline{Y}}} \overline{y}_j \overline{\gamma}_j$ is in the image of $(\phi')_{4n+2}$. As described in Subsection ~\ref{subsec:setupgeneral}, to retract $\Omega S^{4n+2}$ off $A$, it suffices to show that the greatest common divisor of $y_1,\cdots,y_{m_Y},\overline{y}_1,\cdots,\overline{y}_{m_{\overline{Y}}}$ is $1$, and that each $\overline{y}_j$ is even.

\begin{lemma}
    \label{lem:gcdisone}
    The greatest common divisor of $y_1,\cdots,y_{m_Y},\overline{y}_1,\cdots,\overline{y}_{m_{\overline{Y}}}$ is $1$.
\end{lemma}
\begin{proof}
Let $y = (y_1,\cdots,y_{m_{Y}},\overline{y}_1,\cdots,\overline{y}_{m_{\overline{Y}}})$, and suppose $y = dy'$ for some $d > 1$. In this proof, we work with homology with rational coefficients and use the same notation for each element and map to mean its reduction in rational homology.

By definition of $y$, $f_*(a) = y$. Since $g_* \circ f_* = id_*$, $g_*(y) = a$. Let $g_*(y') = a'$ for some $a' \in H_{4n+2}(A;\mathbb{Q})$. By definition of $y'$, $da' = dg_*(y') = g_*(y) = a$. However, integrally, $a$ is a generator of a $\mathbb{Z}$ summand, and so $d=1$ and $a' = a$, which is a contradiction.
\end{proof}

Now we must show that each $\overline{y}_1,\cdots,\overline{y}_{m_{\overline{Y}}}$ is even. To do this, we first determine the rational homotopy type of $A$.

\begin{lemma}
    \label{lem:rathtpyA}
    Let $A$ be a space which retracts off $X \in \prod(\mathcal{P} \cup \mathcal{T})$. Suppose that the only generators in $H_*(A)$ which reduce to primitive generators in rational homology are in degrees of the form $4n+2$, where $n \geq 2$, $n \neq 3$. Then $A$ is rationally homotopy equivalent to a finite type product of loops on spheres of the form $\Omega S^{4n+3}$, where $n \geq 2$, $n \neq 3$.
\end{lemma}
\begin{proof}
    Since $A$ retracts off $X$, $A$ is an $H$-space. Rationally, every $H$-space is homotopy equivalent to a product of spheres and loops on odd dimensional spheres. Since $A$ only contains primitive generators in degrees of the form $4n+2$, $n \geq 2$, $n \neq 3$, there is a rational homotopy equivalence as claimed.
\end{proof}

Using this lemma, we can now show that each $\overline{y}_j$ must be even.

\begin{lemma}
    \label{lem:overlineyiseven}
    For $1 \leq j \leq m_{\overline{Y}}$, $\overline{y}_j$ is even. 
\end{lemma}
\begin{proof}
    By Lemma ~\ref{lem:rathtpyA}, there is a rational homotopy equivalence $A_{\mathbb{Q}} \simeq \prod_{i \in \mathcal{I}} \Omega S^{4n_i+3}_{\mathbb{Q}}$ for some indexing set $\mathcal{I}$. Now localise at the prime $2$, by Proposition ~\ref{prop:cancellation}, there is a $2$-local homotopy equivalence \[A_{(2)} \simeq \prod_{i=1}^l \Omega S^{4n+3}_{(2)} \times S' \times  T,\] where $l \geq 1$, $S'$ is a product of $2$-local loops on spheres of the form $\Omega S^m$, where $m \neq 4n+3$, and $T \in \mathcal{T}$ is a product of indecomposable, $2$-torsion spaces. Recall that $a \in H_{4n+2}(A)$ is a generator reducing to a primitive generator in rational homology, and $f_*(a) = \sum_{i=1}^{m_Y} y_i\gamma_i + \sum_{j=1}^{m_{\overline{Y}}} \overline{y}_j \overline{\gamma}_j + t$, where $t$ has finite order. The Hurewicz theorem implies that there is a map $\nu:S^{4n+2} \rightarrow A_{(2)}$ such that in homology, $\nu$ sends a generator $\lambda$ of $H_{4n+2}(S^{4n+2})$ to $a$. By the universal property of the James construction \cite{J}, this extends to a map $\nu':\Omega S^{4n+3} \rightarrow A_{(2)}$, which sends a generator $\lambda'$ of $H_{4n+2}(\Omega S^{4n+3})$ to $a$. Finally, by the universal property of localisation, there exists a map $\overline{\nu}':\Omega S^{4n+3}_{(2)} \rightarrow A_{(2)}$ which sends $\lambda'$ to $a$. 

    Consider the composite \[\psi: \Omega S^{4n+3}_{(2)} \xrightarrow{\overline{\nu}'} A_{(2)} \xrightarrow{f_{(2)}} X_{(2)} \xrightarrow{\pi_j} \Omega S^{2n_j+2}_{(2)},\] where $\pi_j$ is the projection onto the loops on a sphere corresponding to $\overline{\gamma}_j$, and $n = n_j$. Suppose that $\overline{y}_j$ is odd. By definition, $\psi_*$ sends the generator $\lambda'$ to $\overline{y}_j\overline{\gamma}_j$. Let $J: \Omega S^{2n_j+2} \rightarrow \Omega S^{4n+3}$ be the $2^{nd}$ James-Hopf invariant. Consider the composite $\psi' = J_{(2)} \circ \psi: \Omega S^{4n+3}_{(2)} \rightarrow \Omega S^{4n+3}_{(2)}$. As recounted in \cite[Lemma 2.6]{S} for example, in homology, $J_*$ induces an isomorphism in degree $4n+2$, and so $\psi'_*$ sends $\lambda'$ to $\pm\overline{y}_j \lambda'$. Since $\overline{y}_j$ is odd, $\psi'_*$ is an isomorphism localised at $2$. By \cite[Corollary 5.2]{CPS}, $\Omega S^{4n+3}$ is atomic localised at the prime $2$, meaning that any self-map which is an isomorphism in the bottom non-trivial degree in homology localised at the prime $2$ is a $2$-local homotopy equivalence. The map $\psi'$ factors through $\Omega S^{2n_j+2}_{(2)}$, implying that $\Omega S^{4n+3}_{(2)}$ retracts off $\Omega S^{2n_j+2}_{(2)}$. However, since $n_j \notin \{0,1,3\}$, \cite[Corollary 5.2]{CPS} implies that $\Omega S^{2n_j+2}_{(2)}$ is also atomic, and atomic spaces are indecomposable. Hence, $\overline{y}_j$ must be even.
\end{proof}

Combining Lemma ~\ref{lem:gcdisone} and Lemma ~\ref{lem:overlineyiseven}, $\sum_{i=1}^{m_Y} y_i\gamma_i + \sum_{j=1}^{m_{\overline{Y}}} \overline{y}_j \overline{\gamma}_j$ is an element in the image of $(\phi')_{4n+2}$ as described in Subsection ~\ref{subsec:setupgeneral}. Therefore, we can use the argument in \cite[Subsection 3.3]{S} to show the following.
\begin{lemma}
    \label{lem:case2restractspheres}
    Let $X \in \prod(\mathcal{P} \cup \mathcal{T})$, and $A$ be a space which retracts off $X$. Suppose that the only generators in $H_*(A)$ which reduce to primitive generators in rational homology are in degrees $4n+2$, $n \geq 2$, $n \neq 3$.  There is a homotopy equivalence \[A \simeq P \times T,\] where $P$ is a product of loops on spheres of the forms $\Omega S^{4n+3}$, where $n \geq 2$, $n \neq 3$, and the homology of $T$ is torsion. \qedno
\end{lemma}

\subsection{Conclusion of proof}

Returning to the general case, let $X \in \prod(\mathcal{P} \cup \mathcal{T})$, and $A$ be a space which retracts off $X$. By Lemma ~\ref{lem:case1restractspheres}, there is a homotopy equivalence $A \simeq P \times A'$, where $P$ is a product of spheres and loops on spheres of the forms $S^{n}$, where $n \in \{1,3,7\}$, and $\Omega S^{m+1}$, where $m \in \{2,6,14,4m,2l-1\:|\: m \geq 1,l \notin \{1,2,4\}\}$. Moreover, the only generators in $H_*(A')$ which reduce to primitive generators in rational homology are in degrees of the form $4k+3$, where $k \geq 2$ and $k \neq 3$. Lemma ~\ref{lem:case2restractspheres} then implies there is a homotopy equivalence  $A' \simeq P' \times T$ where $P'$ is a product of loops on spheres of the forms $\Omega S^{4n+3}$, where $n \geq 2$, $n \neq 3$, and the homology of $T$ is torsion. Combining these, we obtain a homotopy equivalence \[A \simeq P \times P' \times T,\] where $P,P' \in \prod\mathcal{P}$, and the homology of $T$ is torsion. Since $A$ retracts off $X$, $T$ retracts off $X$, and Theorem ~\ref{thm:Tclosedunderretracts} implies that $T \in \prod\mathcal{T}$. Summarising, we have obtained the following result.

\begin{theorem}
    \label{thm:PcupTclosedunderretracts}
    Let $X \in \prod(\mathcal{P} \cup \mathcal{T})$ and $A$ be a space which retracts off $X$. Then $A \in \prod(\mathcal{P} \cup \mathcal{T})$. \qedno
\end{theorem}

\section{Preliminary decompositions of Moment-angle complexes}
\label{subsec:relationsbetweenWandP}

In this section, we prove some relations between spaces in $\bigvee(\mathcal{W} \cup \mathcal{M})$ and $\prod(\mathcal{P} \cup \mathcal{T})$. These will be generalisations of the relations between spaces in $\bigvee\mathcal{W}$ and $\prod\mathcal{P}$ shown in \cite[Subsection 2.5]{S}. Before proving these relations, we require a result of \cite[Corollary 6.6]{N3}, which gives a wedge decomposition for certain smash products of Moore spaces.

\begin{lemma}
    \label{lem:smashofMoorespaces}
    Let $p$ and $q$ be primes, and $r,s \geq 1$ such that $\max\{p^r,q^s\} > 2$. If $p \neq q$, then $P^n(p^r) \wedge P^m(q^s)$ is contractible. If $p = q$, there is a homotopy equivalence \[P^n(p^r) \wedge P^m(p^s) \simeq P^{n+m}(p^{\min\{r,s\}}) \vee P^{n+m-1}(p^{\min\{r,s\}}).\] \qed
\end{lemma}

\begin{lemma}
    \label{lem:suspensionofTisinW'}
    The following hold:
    \begin{enumerate}
        \item let $A$ be a space in $\prod\mathcal{T}$, then $\Sigma A \in \bigvee\mathcal{M}$;
        \item let $A \in \bigvee(\mathcal{W} \cup \mathcal{M})$, then $\Omega A \in \prod(\mathcal{P} \cup \mathcal{T})$;
        \item let $A \in \prod(\mathcal{P} \cup \mathcal{T})$, then $\Sigma A \in \bigvee(\mathcal{W} \cup \mathcal{M})$;
        \item let $X$ be a space such that $\Sigma X \in \bigvee(\mathcal{W} \cup \mathcal{M})$, and let $A_1,\cdots,A_m$ be spaces in $\prod(\mathcal{P} \cup \mathcal{T})$, then $\Sigma (X \wedge A_1 \wedge \cdots A_m) \in \bigvee(\mathcal{W} \cup \mathcal{M})$;
        \item let $X$ and $Y$ are spaces such that $\Sigma X \in \bigvee(\mathcal{W} \cup \mathcal{M})$, and $\Omega Y \in \prod(\mathcal{P} \cup \mathcal{T})$, then we obtain $\Omega (X \ltimes Y) \in \prod(\mathcal{P} \cup \mathcal{T})$;
        \item let $X_1,\cdots,X_m$ be spaces such that $\Omega X_i \in \prod(\mathcal{P} \cup \mathcal{T)}$, then $\Omega (\bigvee_{i=1}^m X_i) \in \prod(\mathcal{P} \cup \mathcal{T})$.
    \end{enumerate}
\end{lemma}
\begin{proof}
    For part $(1)$, since $A \in \prod\mathcal{T}$, there is a homotopy equivalence $A \simeq \prod_{i \in \mathcal{I}} T_i$, where each $T_i$ is an indecomposable space in the loop space decomposition of $\Omega (\bigvee_{j \in \mathcal{J}} P^{n_i}(p_i^{r_i}))$. In particular, $\Sigma T_i$ retracts off $\Sigma \Omega (\bigvee_{j \in \mathcal{J}} P^{n_i}(p_i^{r_i}))$. Since each $P^{n_i}(p^r)$ is a suspension, the James splitting implies that there is a homotopy equivalence \begin{equation}\label{eqn:JameswedgeofMoore}\Sigma \Omega (\bigvee\limits_{j \in \mathcal{J}} P^{n_i}(p_i^{r_i})) \simeq \bigvee\limits_{j \geq 1} (\bigvee\limits_{j \in \mathcal{J}} P^{n_i}(p_i^{r_i}))^{\wedge j}.\end{equation} Distributing the wedge sum over the smash, we obtain a wedge of spaces which are smashes of Moore spaces, where by Lemma ~\ref{lem:smashofMoorespaces}, each wedge summand consists of Moore spaces of a fixed prime. By Lemma ~\ref{lem:smashofMoorespaces}, if $p^r \neq 2$, then the smash products decompose further as a wedge of Moore spaces. If $p^r = 2$, then by Proposition ~\ref{prop:uniquedecomp}, this decomposes as a finite type wedge of indecomposable spaces which appear in the unique decomposition of smashes of mod $2$ Moore spaces. Therefore, by Theorem ~\ref{thm:WcupMclosedunderretracts} and by definition of $\mathcal{M}$, $\Sigma T_i \in \bigvee(\mathcal{W} \cup \mathcal{M})$. Hence, iterating the homotopy equivalence $\Sigma (X \times Y) \simeq \Sigma X \vee \Sigma Y \vee \Sigma (X \wedge Y)$, and shifting the suspension coordinate, we obtain that $\Sigma A \in \bigvee\mathcal{M}$.

    For part $(2)$, write $A$ as \[A \simeq \bigvee\limits_{i\in \mathcal{I}} S^{n_i} \vee \bigvee\limits_{j \in \mathcal{J}} P^{m_j}(p_j^{r_j}).\] The Hilton Milnor theorem implies there is a homotopy equivalence \[\Omega A \simeq \prod\limits_{k \in \mathcal{K}} \Omega \Sigma (A_{1} \wedge \cdots \wedge A_k),\] where $\mathcal{K}$ is some indexing set, and each $A_i$ is either a sphere or a Moore space. 
    
    Consider the term $A'_k = \Omega \Sigma (A_{1} \wedge \cdots \wedge A_k)$. If each $A_l$ is a sphere, then $A'_k$ is homeomorphic to the loops on a sphere. If there exists an $A_l$ which is a Moore space, then by Lemma ~\ref{lem:smashofMoorespaces}, $A'_k$ is either contractible, the loops on a wedge of Moore spaces of the form $P^{n_j}(p^r)$ for $n_j \geq 3$, $p$ a fixed prime, and $r \geq 1$ fixed, or the smash product of mod $2$ Moore spaces. In the first two cases, it is clear by definition of $\mathcal{T}$ that $A'_k \in \prod\mathcal{T}$. For the latter case, consider $\Omega \Sigma (P^{n_1}(2) \wedge \cdots \wedge P^{n_l}(2))$, where $l \geq 2$, and $n_i \geq 2$. The Hilton-Milnor thoerem implies there is a homotopy equivalence \[\Omega\Sigma \left(\bigvee\limits_{i=1}^l P^{n_i}(2)\right) \simeq \prod\limits_{i \in B}\Omega\Sigma(P^{n_1}(2)^{\wedge{b(1)}} \wedge \cdots \wedge P^{n_l}(2)^{\wedge b(l)}),\] where $B$ is a Hall basis of the free ungraded Lie algebra on $\{1,\cdots,l\}$ over $\mathbb{Z}$, and $b(i)$ is the number of times $i$ appears in the bracket $b$. By construction of a Hall basis (see \cite[p.120]{N1} for example), the smash product $\Omega\Sigma(P^{n_1}(2) \wedge \cdots \wedge P^{n_l}(2))$ appears as a product term, and so $\Omega\Sigma (P^{n_1}(2) \wedge \cdots \wedge P^{n_l}(2))$ retracts off $\Omega\Sigma (\bigvee_{i=1}^l P^{n_i}(2))$. Therefore, Theorem ~\ref{thm:Tclosedunderretracts} implies $\Omega\Sigma (P^{n_1}(2) \wedge \cdots \wedge P^{n_l}(2)) \in \prod\mathcal{T}$. Hence, each $A'_k \in \prod\mathcal{T}$, and so we obtain that $\Omega A \in \prod(\mathcal{P} \cup \mathcal{T})$.

    For part $(3)$, write $A$ as \[A \simeq \prod\limits_{i\in \mathcal{I}} S^{n_i} \times \prod\limits_{j \in \mathcal{J}} \Omega S^{m_i} \times \prod\limits_{k \in \mathcal{K}} T_i,\] where $T_i \in \mathcal{T}$. Iterating the homotopy equivalence $\Sigma(X \times Y) \simeq \Sigma X \vee \Sigma Y \vee \Sigma (X \vee Y)$, we obtain a homotopy equivalence \begin{equation}\label{eqn:sigmaAasawedge}\Sigma A \simeq \bigvee\limits_{l \in \mathcal{L}} \Sigma (A_1 \wedge \cdots \wedge A_l),\end{equation} where each $A_i$ is either a sphere, the loops on a sphere, or some $T_j$. Consider $A'_l = \Sigma (A_1 \wedge \cdots \wedge A_l)$. By the James construction, $\Sigma \Omega S^n \in \bigvee\mathcal{W}$, and by part (a), each $\Sigma T_k \in \bigvee\mathcal{M}$. By shifting the suspension coordinate, we can decompose $A'_l$ as a wedge of spaces $W'_{l'}$, where each $W'_{l'}$ is the suspension of a smash product of spheres and Moore spaces. If each space is a sphere, then $W'_l$ is a sphere. If there is a Moore space, then by Lemma ~\ref{lem:smashofMoorespaces}, $W_{l'} $ can be decomposed further as a wedge of spaces, each of which is either a Moore space, or a smash of mod $2$ Moore spaces. In either case, by definition of $\mathcal{M}$, $W_{l'} \in \bigvee\mathcal{M}$, and so $A'_l \in \bigvee(\mathcal{W} \cup \mathcal{M})$. Therefore, by \eqref{eqn:sigmaAasawedge}, $A \in \bigvee(\mathcal{W} \cup \mathcal{M})$.

    The remaining parts follow by arguing as in \cite[Lemma 2.13, Lemma 2.14, Corollary 2.16]{S} respectively.\end{proof}

The next result shows the property of $\Omega \zk \in \prod(\mathcal{P} \cup \mathcal{T})$ is closed under pushouts of simplicial complexes over full subcomplexes.

\begin{theorem}
\label{thm:pushoutofPisinP}
    Let $K$ be a simplicial complex defined as the pushout \[\begin{tikzcd}
	L & {K_1} \\
	{K_2} & K
	\arrow[from=1-1, to=1-2]
	\arrow[from=1-2, to=2-2]
	\arrow[from=1-1, to=2-1]
	\arrow[from=2-1, to=2-2]
\end{tikzcd}\] where either $L = \emptyset$ or $L$ is a proper full subcomplex of $K_1$ and $K_2$. If $\Sigma A_i \in \bigvee(\mathcal{W} \cup \mathcal{M})$ for all $i$, $\Omega \caa^{K_1} \in \prod(\mathcal{P} \cup \mathcal{T})$ and $\Omega \caa^{K_2} \in \prod(\mathcal{P} \cup \mathcal{T})$, then $\Omega \caa^K \in \prod(\mathcal{P} \cup \mathcal{T})$.
\end{theorem}
\begin{proof}
    The case where $\Sigma A_i \in \bigvee\mathcal{W}$, and each loop space is in $\prod\mathcal{P}$ was proved in \cite[Theorem 4.1]{S}. The proof depended on two results: closure of $\prod\mathcal{P}$ under retracts and the $\bigvee\mathcal{W}$ analogue of Lemma ~\ref{lem:suspensionofTisinW'}. The same argument holds for $\prod(\mathcal{P} \cup \mathcal{T})$ using Theorem ~\ref{thm:PcupTclosedunderretracts}, and Lemma ~\ref{lem:suspensionofTisinW'}.
\end{proof}

Using Theorem ~\ref{thm:pushoutofPisinP}, we can obtain a generalisation of \cite[Theorem 1.1]{S}.  To do this, we require the following pushout decomposition for a simplicial complex $K$ from \cite[Lemma 4.4]{S}. 

\begin{lemma}
\label{decomposesimpcomp}
    Let $K$ be a simplicial complex and $v \in V(K)$. Then $K$ can be written as the pushout \[\begin{tikzcd}
	{K_{N(v)}} & {K_{v \cup N(v)}} \\
	{K_{V(K) \setminus \{v\}}} & K.
	\arrow[from=2-1, to=2-2]
	\arrow[from=1-2, to=2-2]
	\arrow[from=1-1, to=1-2]
	\arrow[from=1-1, to=2-1]
\end{tikzcd}\] Moreover, $K_{N(v)}$ is a full subcomplex of both $K_{v \cup N(v)}$ and $K_{V(K) \setminus \{v\}}$. \qedno
\end{lemma}
For a simplicial complex $K$, let $C_K$ be the collection of full subcomplexes of $K$ whose $1$-skeleton has no missing edges. For a vertex $v \in V(K)$, denote by $N(v)$ the set of vertices adjacent to $v$ in the $1$-skeleton of $K$. A \textit{dominating vertex} of $K$ is a vertex $v$ such that $N(v) = V(K) \setminus \{v\}$. In other words, $v$ is adjacent to every other vertex in the $1$-skeleton of $K$. The following result is a generalisation of \cite[Theorem 1.1]{S}

\begin{theorem}
    \label{thm:inPcupTiffullsubcompleteskel}
    If $K$ is a simplicial complex, then $\Omega \zk \in \prod(\mathcal{P} \cup \mathcal{T})$ if and only if $\Omega \mathcal{Z}_{K_I} \in \prod(\mathcal{P} \cup \mathcal{T})$ for all $K_I \in C_K$. Moreover, $\Omega \zk \in \prod\mathcal{P}$ if and only if $\Omega \mathcal{Z}_{K_I} \in \prod\mathcal{P}$ for all $K_I \in C_K$. 
\end{theorem}

\begin{proof}
Suppose $\Omega \zk \in \prod(\mathcal{P} \cup \mathcal{T})$ but there exists $K_I \in C_K$ such that $\Omega \mathcal{Z}_{K_I} \notin \prod(\mathcal{P} \cup \mathcal{T})$. By \cite{DS}, $\mathcal{Z}_{K_I}$ retracts off $\zk$, and so $\Omega \mathcal{Z}_{K_I}$ retracts off $\Omega \zk$. Theorem ~\ref{thm:PcupTclosedunderretracts} then implies that $\Omega \mathcal{Z}_{K_I} \in \prod(\mathcal{P} \cup \mathcal{T})$ which is a contradiction.

Now suppose that $\Omega \mathcal{Z}_{K_I} \in \prod(\mathcal{P} \cup \mathcal{T})$ for all $K_I \in C_K$. We proceed by strong induction. If $K$ has one vertex, then $\zk$ is contractible, and so $\Omega \zk \in \prod(\mathcal{P} \cup \mathcal{T})$.

    Now suppose $K$ has $m$ vertices, and the result is true for all $n <m$. If every vertex of $K$ is a dominating vertex, then every vertex in $K$ is adjacent to every other vertex in $K$. Hence, $K \in C_K$ so by assumption, $\Omega \zk \in \prod(\mathcal{P} \cup \mathcal{T})$. Therefore, suppose there exists a vertex $v \in V(K)$ such that $v$ is not a dominating vertex of $K$. By Lemma ~\ref{decomposesimpcomp}, $K$ can be written as the pushout \[\begin{tikzcd}
	{K_{N(v)}} & {K_{v \cup N(v)}} \\
	{K_{V(K) \setminus \{v\}}} & K
	\arrow[from=2-1, to=2-2]
	\arrow[from=1-2, to=2-2]
	\arrow[from=1-1, to=1-2]
	\arrow[from=1-1, to=2-1]
\end{tikzcd}\] where $K_{N(v)}$ is a full subcomplex of both $K_{v \cup N(v)}$ and $K_{V(K) \setminus \{v\}}$. Since $v$ is not a dominating vertex, $v \cup N(v)$ is not the whole vertex set of $K$, so $K_{v \cup N(v)}$ and $K_{V(K) \setminus \{v\}}$ are simplicial complexes with strictly fewer vertices than $K$. Therefore, by the inductive hypothesis, $\Omega \mathcal{Z}_{K_{v \cup N(v)}} \in \prod(\mathcal{P} \cup \mathcal{T})$ and $\Omega \mathcal{Z}_{K_{V(K) \setminus \{v\}}} \in \prod(\mathcal{P} \cup \mathcal{T})$. Hence, Theorem ~\ref{thm:pushoutofPisinP} implies that $\Omega \zk \in \prod(\mathcal{P} \cup \mathcal{T})$. 

In the special case where $\Omega \zk \in \prod\mathcal{P}$, the same proof follows through.
\end{proof}

\section{Loop spaces of moment-angle complexes associated to \texorpdfstring{$2$}{2}-dimensional simplicial complexes}

In this section, we consider moment-angle complexes associated to $2$-dimensional simplicial complexes. We start with a more general statement. To prove this, we require the following from \cite[Theorem 1.6]{IK}. Recall that a simplicial complex $K$ is called $k$-neighbourly if any $I \subseteq [m]$ with $|I| \leq k+1$ spans a simplex. For $x \in \mathbb{R}$, define $\ceil{x}$ to be the smallest integer $z$ such that $z \geq x$.

\begin{lemma}
    \label{lem:neighbourlyBBCG}
    Let $K$ be a simplicial complex on $[m]$. If $K$ is $\ceil{\frac{\mathrm{dim}(K)}{2}}$-neighbourly, then there is a homotopy equivalence \[\zk \simeq \bigvee\limits_{I \notin K} \Sigma^{1+|I|} |K_I|.\eqno\qedno\]
\end{lemma}

We also require a wedge decomposition of $(n-1)$ connected, $(n+1)$-dimensional $CW$-complexes, where $n \geq 2$. This is proved in \cite[Example 4C.2]{H} for example.

\begin{lemma}
\label{lem:lowdimdecompwedgeofsphereandMoore}
    Let $X$ be an $(n-1)$ connected, $(n+1)$-dimensional $CW$-complex, where $n \geq 2$. Then $X$ is homotopy equivalent to a wedge of spheres and Moore spaces. If $H_*(X)$ is torsion free, then $X \in \bigvee\mathcal{W}$. \qed
\end{lemma}

\begin{theorem}
    \label{thm:highlyneighbourly}
    If $K$ is a simplicial complex on $[m]$ such that each $K_I \in C_K$ is $\mathrm{dim}(K_I)-1$-neighbourly, then $\Omega \zk \in \prod(\mathcal{P} \cup \mathcal{T})$. If $H_*(|K_I|;\mathbb{Z})$ is torsion free for all $K_I \in C_K$, then $\Omega \mathcal{Z}_{K} \in \prod\mathcal{P}$.
\end{theorem}
\begin{proof}
By Theorem ~\ref{thm:inPcupTiffullsubcompleteskel}, it suffices to show that $\Omega Z_{K_I} \in \prod(\mathcal{P} \cup \mathcal{T})$, where $K_I \in C_K$. If $\mathrm{dim}(K_I) =1$, then $K_I$ is the $1$-skeleton of a simplex, and so by \cite[Theorem 9.1]{GT}, $\mathcal{Z}_{K_I}$ is homotopy equivalent to a wedge of simply connected spheres. The Hilton-Milnor theorem then implies that $\Omega \mathcal{Z}_{K_I} \in \prod\mathcal{P}$. If $\mathrm{dim}(K_I) > 1$, then $\mathrm{dim}(K_I)-1 \geq \ceil{\frac{\mathrm{dim}(K_I)}{2}}$, and so Lemma ~\ref{lem:neighbourlyBBCG} implies that there is a homotopy equivalence \[\mathcal{Z}_{K_I} \simeq \bigvee\limits_{J \notin K_I} \Sigma^{1+|J|} |K_J|.\] Since $K_I$ is $\mathrm{dim}(K_I)-1$-neighbourly, each $K_J$ is $\mathrm{dim}(K_I)-1$-neighbourly. Moreover, each $K_J$ is a full subcomplex of $K_I$, and so $\mathrm{dim}(K_J) \leq \mathrm{dim}(K_I)$. Hence, each $K_J$ is at least $\mathrm{dim}(K_J)-1$ neighbourly. By \cite[Lemma 10.8]{IK} for example, this implies that each $K_J$ is at least $(\mathrm{dim}(K_J)-2)$-connected. Since $\Sigma^{1+|J|} |K_J|$ is simply connected, by Lemma ~\ref{lem:lowdimdecompwedgeofsphereandMoore}, each $\Sigma^{1+|J|} |K_J|$ is homotopy equivalent to a wedge of spheres and Moore spaces. Hence part $(2)$ of Lemma ~\ref{lem:suspensionofTisinW'} implies that $\Omega \Sigma^{1+|J|}|K_J| \in \prod(\mathcal{P} \cup \mathcal{T})$, and therefore part $(6)$ of Lemma ~\ref{lem:suspensionofTisinW'} implies that $\Omega \mathcal{Z}_{K_I} \in \prod(\mathcal{P} \cup \mathcal{T})$. Hence, $\Omega \zk \in \prod(\mathcal{P} \cup \mathcal{T})$.

    If $H_*(|K_I|;\mathbb{Z})$ is torsion free for all $K_I \in C_K$, then by Lemma ~\ref{lem:lowdimdecompwedgeofsphereandMoore}, each $\Sigma^{1+|J|} |K_J|$ is homotopy equivalent to a wedge of spheres. The Hilton-Milnor theorem implies that $\Omega \Sigma^{1+|J|}|K_J| \in \prod\mathcal{P}$, and so Theorem ~\ref{thm:inPcupTiffullsubcompleteskel} implies that $\Omega \zk \in \prod\mathcal{P}$.
\end{proof}

Theorem ~\ref{thm:inPcupTiffullsubcompleteskel} and Theorem ~\ref{thm:highlyneighbourly} can be applied to give a coarse decomposition of the loops on a moment-angle complex associated to any $2$-dimensional simplicial complex.

\begin{theorem}
    \label{thm:2dimloop}
    Let $K$ be a $2$-dimensional simplicial complex. Then $\Omega \mathcal{Z}_{K} \in \prod(\mathcal{P} \cup \mathcal{T})$.
\end{theorem}
\begin{proof}
    If $K_I \in C_K$ is $1$-dimensional, then by \cite[Theorem 9.1]{GT}, $\mathcal{Z}_{K_I}$ is homotopy equivalent to a wedge of simply connected spheres. The Hilton-Milnor theorem implies that $\Omega \mathcal{Z}_{K_I} \in \prod\mathcal{P}$, which is contained in $\prod(\mathcal{P} \cup \mathcal{T})$. If $K_I \in C_K$ is $2$-dimensional, then it is $1$-neighbourly and Theorem ~\ref{thm:highlyneighbourly} implies that $\Omega \mathcal{Z}_{K_I} \in \prod(\mathcal{P} \cup \mathcal{T})$. Since each $K_I \in C_K$ is $1$-dimensional or $2$-dimensional, Theorem ~\ref{thm:inPcupTiffullsubcompleteskel} implies that $\Omega \zk \in \prod(\mathcal{P} \cup \mathcal{T})$.
\end{proof}

As an example of Theorem ~\ref{thm:2dimloop} which contains torsion, let $K$ be the $6$-vertex triangulation of $\mathbb{R}P^2$. In this case, $K$ is a $1$-neighbourly simplicial complex whose homology contains $2$-torsion, and so the decomposition in Theorem ~\ref{thm:2dimloop} will contain indecomposable $2$-torsion spaces. In \cite{LMR}, $1$-neighbourly, $2$-dimensional simplicial complexes which have arbitrarily large torsion in homology are constructed. Therefore, the loop space of a moment-angle complex associated to a $2$ dimensional simplicial complex can contain arbitrarily large torsion in homology. In the case that each $|K_I|$ has torsion free homology, where $K_I \in C_K$, we obtain the following.

\begin{corollary}
    \label{cor:torsionfree}
    Let $K$ be a $2$-dimensional simplicial complex. Suppose that $H_*(|K_I|;\mathbb{Z})$ is torsion free for all $K_I \in C_K$. Then $\Omega \mathcal{Z}_{K} \in \prod\mathcal{P}$. \qedno
\end{corollary}

\section{Loop space decompositions of certain moment-angle complexes after localisation}
\label{subsec:localisedresult}

In this section, we use the argument in Theorem ~\ref{thm:highlyneighbourly} to show that for a simplicial complex $K$ under certain conditions on the full subcomplexes $K_I \in C_K$, there is a finite set of primes $P$ for which $\Omega \zk \in \prod\mathcal{P}$ localised away from $P$. This set of primes is controlled purely by the underlying simplicial complex. Recall that a simplicial complex $K$ is rationally Golod if all products and higher Massey products in $H^*(\zk;\mathbb{Q})$ are trivial.

\begin{theorem}
\label{thm:localisedprodofspheres}
    Let $K$ be a simplicial complex such that $K_I$ is rationally Golod for all $K_I \in C_K$. For $K_I \in C_K$, let $K_I$ be $k_I$-neighbourly, and let $M = \max_{K_I \in C_K}\{|I|+dim(K_I)-2k_I\}$. Localise away from primes $p \leq \frac{1}{2}M$ and primes $p$ appearing as $p$-torsion in $H_*(|K_I|;\mathbb{Z})$ for any $K_I \in C_K$. Then $\Omega \mathcal{Z}_K \in \prod\mathcal{P}$.
\end{theorem}
\begin{proof}
        First, note that Theorem ~\ref{thm:inPcupTiffullsubcompleteskel} holds $p$-locally for any prime $p$. By Theorem ~\ref{thm:inPcupTiffullsubcompleteskel}, it suffices to show $\Omega \mathcal{Z}_{K_I} \in \prod\mathcal{P}$ for each $K_I \in C_K$, after localisation away from primes $p \leq \frac{1}{2}M$ and primes $p$ appearing as $p$-torsion in $H_*(|K_I|;\mathbb{Z})$ for any $K_I \in C_K$. Consider $\mathcal{Z}_{K_I}$, where $K_I \in C_K$. By assumption, $K_I$ is rationally Golod. Therefore, by Proposition ~\ref{prop:rationalgolodmeanssuspension} and Proposition ~\ref{prop:suspensionsplitting}, there is a rational homotopy equivalence \[\mathcal{Z}_{K_I} \simeq \bigvee\limits_{J \subseteq I} \Sigma^{1+|J|} |K_J|.\] Since each wedge summand is a suspension, $\mathcal{Z}_{K_I}$ is rationally a wedge of spheres. It can be shown using Proposition ~\ref{prop:suspensionsplitting} that $\mathcal{Z}_{K_I}$ is $(2k_I+2)$-connected, and has dimension $1+|I|+dim(K_I)$. Therefore, by Lemma ~\ref{lem:plocallywedge}, localised away from \[p \leq \frac{1}{2}(1+|I|+dim(K_I)-(2k_I+2)+1) = \frac{1}{2}(|I|+dim(K_I)-2k_I) \leq \frac{1}{2}M\] and those primes $p$ appearing as $p$-torsion in $H_*(|K_J|;\mathbb{Z})$, $\mathcal{Z}_{K_I}$ is homotopy equivalent to a wedge of spheres. The Hilton-Milnor theorem then implies that $\Omega \mathcal{Z}_{K_I} \in \prod\mathcal{P}$. Repeating this argument for each $K_I \in C_K$, by Lemma ~\ref{thm:inPcupTiffullsubcompleteskel}, $\Omega \mathcal{Z}_{K_I} \in \prod\mathcal{P}$ when localised away from primes $p \leq \frac{1}{2}M$ and any primes appearing as $p$-torsion in $H_*(|K_I|;\mathbb{Z})$ for each $K_I \in C_K$.
\end{proof}

Theorem ~\ref{thm:localisedprodofspheres} has interesting connections to a problem posed by McGibbon and Wilkerson. Recall from the Introduction that a space $X$ is called \textit{rationally elliptic} if it has finitely many rational homotopy groups. Otherwise, it is called \textit{rationally hyperbolic}.  McGibbon and Wilkerson \cite{MW} showed that any finite, simply connected, rationally elliptic space $X$ has $\Omega X \in \prod\mathcal{P}$, after localising at a sufficiently large prime. A consequence of a decomposition of this form pointed out is that the Steenrod algebra acts trivially on $H_*(\Omega X;\mathbb{Z}/p\mathbb{Z})$ at almost all primes $p$. They asked the extent to which this holds for rationally hyperbolic spaces. Theorem ~\ref{thm:localisedprodofspheres} gives an analogous result for the moment-angle complexes in Theorem ~\ref{thm:localisedprodofspheres}.
\begin{corollary}
    \label{cor:Steenrodalgebra}
    Let $K$ be a simplicial complex such that $K_I$ is rationally Golod for all $K_I \in C_K$. Then at all but finitely many $p$, the Steenrod algebra acts trivially on $H^*(\Omega \zk;\mathbb{Z}/p\mathbb{Z})$. \qed
\end{corollary}

One particular family of examples is when $K$ is a $2$-dimensional simplicial complex, Lemma ~\ref{lem:neighbourlyBBCG} implies that every $K_I \in C_K$ is rationally Golod, and so we obtain the following. \begin{corollary}
    \label{cor:2dimSteenrodalgebra}
    Let $K$ be a $2$-dimensional simplicial complex. Then at all but finitely many primes $p$, the Steenrod algebra acts trivially on $H^*(\Omega \zk;\mathbb{Z}/p\mathbb{Z})$. \qed
\end{corollary}

It was shown in \cite{BBCG2} that the moment-angle complex associated to $K$ is rationally elliptic if and only if the minimal missing faces of $K$ are mutually disjoint. In particular, if $K$ is a simplicial complex such that there is a vertex which is not adjacent to two other vertices, then $\zk$ is rationally hyperbolic. Corollary ~\ref{cor:2dimSteenrodalgebra} can be used to generate infinitely many rationally hyperbolic examples for which the answer to the question of McGibbon and Wilkerson is affirmative. One other point to note is that the set of primes for which the decomposition of rationally elliptic spaces given by McGibbon and Wilkerson holds is not explicit. The conditions on $K$ in Theorem ~\ref{thm:localisedprodofspheres} give an explicit set of primes for which such a decomposition holds for $\Omega \zk$. Also, this set of primes can certainly be enlarged depending on the combinatorics of $K$. For example, if $K$ is a $2$-dimensional simplicial complex such that $H_*(|K_I|;\mathbb{Z})$ is torsion free for all $K_I \in C_K$, Corollary ~\ref{cor:torsionfree} implies that such a decomposition holds integrally.

\bibliographystyle{amsalpha}

\end{document}